\documentclass[10pt,a4paper]{amsart}

\usepackage{amscd}
\usepackage{setspace}
\usepackage{a4wide}
\usepackage{amsmath,amsthm}
\usepackage{amssymb, amsfonts}
\usepackage[dvips]{graphicx}
\usepackage[margin=.45in, font=footnotesize,labelfont=bf,labelsep=period]{caption}

\doublespace

\newtheorem{theorem}{Theorem}[section]

\newtheorem{lemma}[theorem]{Lemma}
\newtheorem{proposition}[theorem]{Proposition}

\newtheorem{remark}{Remark}[section]

\newcommand{\argmax}{\operatornamewithlimits{argmax}}
\newcommand{\Exp}{\operatornamewithlimits{Exp}}
\newcommand{\N}{\ensuremath{{\mathbb N}}}
\newcommand{\E}{\ensuremath{{\mathbf E}}}

\usepackage[final]{changes} 

\allowdisplaybreaks

\numberwithin{equation}{section}

\title{Resolvent-techniques for multiple exercise problems}
\author{S\"{o}ren Christensen}
\author{Jukka Lempa}
\address[S\"{o}ren Christensen]{Mathematical Institute, Christian-Albrechts-University in Kiel, Ludewig-Meyn-Str. 4, D -- 24098 Kiel, e-mail:\texttt{christensen@math.uni-kiel.de}}
\address[Jukka Lempa]{School of business, Faculty of Social Sciences, Oslo and Akershus University College, P.O. Box 4 St. Olavs plass, NO -- 0130 Oslo, e-mail: \texttt{jukka.lempa@hioa.no}}

\begin{document}

\begin{abstract}
We study optimal multiple stopping of strong Markov processes with random refraction periods. The refraction periods are assumed to be exponentially distributed with a common rate and independent of the underlying dynamics. Our main tool is using the resolvent operator. In the first part, we reduce infinite stopping problems to ordinary ones in a general strong Markov setting. This leads to explicit solutions for wide classes of such problems. Starting from this result, we analyze problems with finitely many exercise rights and explain solution methods for some classes of problems with underlying L\'evy and diffusion processes, where the optimal characteristics of the problems can be identified more explicitly. We illustrate the main results with explicit examples.
\end{abstract}

\maketitle

\section{Introduction}


In this paper, we investigate the following optimal multiple stopping problem. Assume that the agent follows the evolution of a continuous-time state variable $X$ and has the possibility to choose $N$ stopping times such that there are exponentially distributed refraction periods between the stopping times. At each stopping time, the agent gets a payoff contingent on the state of $X$ and her objective is to maximize the expected present value of the total payoff.



Optimal multiple stopping problems with deterministic refraction periods have been studied quite extensively over the recent years. In particular, theory of Snell envelopes is understood well in both discrete and continuous time, see \cite{Car, Touzi, Alexandrov_Hambly, BenderDual, BenderDual2, Meinshausen, Schoenmakers, Zeghal_Mnif}.  In the recent study \cite{CIJ}, the authors develop a general approach for optimal multiple stopping with random refraction periods. In this paper, the refraction periods are assumed to be almost surely finite, non-negative random variables. Furthermore, the exercise policies are formalized with respect to appropriately extended filtrations, see \cite{CIJ}, Section 2, for details. From the mathematical point of view, our study is concerned with a special case of \cite{CIJ}. However, by imposing the additional assumption of the distribution of the refraction periods, we can exploit efficiently the known connection between exponentially distributed random times and the resolvent semigroup in our study. This allows us to get a much more detailed picture on the problem. For instance, we can solve the problem explicitly in many cases.


One of the major drawbacks in most of the previously mentioned articles is that it is hard to find examples that allow for an explicit solution, even in the case of simple one-dimensional underlying diffusion processes. Indeed, when studying optimal multiple stopping with deterministic refraction periods, the main obstacle is how to determine the value of the remaining stopping opportunities, the so-called \emph{continuation value}. In our framework with exponentially distributed refraction periods, we can tackle this using the theory of resolvent semigroups. From applications point of view, it is also interesting to study multiple stopping with random refraction periods. As we already mentioned, the majority of the literature is concerned with deterministic refraction periods. If we consider applications in real investment problems, it can very well be that the waiting times between exercises of a multi-strike real option are random. As an example in the real option spirit, consider a firm facing the timing problem of multiple identical investment projects. The firm has the capability to execute only one project at a time and it takes a period of time to complete the project. Here, the refraction period has the interpretation as \emph{time to build}, see, e.g., \cite{Majd_Pindyck}. Depending on nature of the project, it can be very difficult to tell in advance how long it takes to complete the project and therefore it can be reasonable to include uncertain times to build into the investment problem. Our study gives a flexible yet tractable model to study problems of this form.



This paper contributes to the theory of optimal multiple stopping in the following ways. 
We present a thorough analysis of the case of exponential waiting times which does not appear in the literature before. We point out an interesting connection on the multiple stopping problem to a class of stochastic impulse control problems. We give separate account of both the case with infinite and finite number of stopping times, we call these problems \emph{infinite} and \emph{finite} stopping problems, respectively. The infinite stopping problem, which is also interesting in its own right, gives an important point of reference to the finite stopping problem as a natural limiting case. Our main results are presented for general underlying strong Markov dynamics. The general results are complemented with a number of more specific results for underlying L\'evy and diffusion dynamics. The results are also illustrated with several explicit examples.

The reminder of the paper is organized as follows. In Section \ref{sec:setting} we formalize the optimal multiple stopping problem, whereas in Section \ref{sec:infinite} we first concentrate on the infinite multiple stopping problem, where the issuer has infinitely many exercise rights. Surprisingly, it turns out that these problems can be reduced to ordinary optimal stopping problems and therefore can be solved explicitly in many cases of interest. Some of these examples are given in Section \ref{sec:ex_infinite}, including multidimensional underlying processes and processes with jumps. In Section \ref{sec:mult} we study some of the general properties of stopping problems with finitely many exercise rights. Section \ref{sec:closed_form} is devoted to the analysis of some specific classes of the problems with underlying diffusion dynamics and L\'evy processes. The study is wrapped up with an explicit example of the finite stopping problem in Section \ref{sec:call}.

\section{The Multiple Stopping Problem}\label{sec:setting}

We start by laying down the probabilistic foundation for the optimal multiple stopping problem. Let $(\Omega,\mathcal{F},\mathbb{F},\mathbf{P})$ be a complete filtered probability space satisfying the usual conditions, where $\mathbb{F}=\{\mathcal{F}_t\}_{t\geq0}$, see \cite{BS}, p. 2. We assume that the underlying $X$ is a strong Markov process defined on $(\Omega,\mathcal{F},\mathbb{F},\mathbf{P})$ and taking values in $E\subseteq\mathbf{R}^d$ for some $d\geq1$ with the initial state $x\in E$. 
As usual, we augment the state space $E$ with a topologically isolated element $\Delta$ if the process $X$ is non-conservative. Then the process $X$ can be made conservative on the augmented state space $E^{\Delta}:=E\cup\{\Delta\}$, see \cite{Rogers_Williams_I}. In what follows, we drop the superscript $\Delta$ from the notation. By convention, we augment the definition of functions $g$ on $E$ with $g(\Delta)=0$. 

Denote as $\mathbf{P}_x$ the probability measure $\mathbf{P}$ conditioned on the initial state $x$ and as $\mathbf{E}_x$ the expectation with respect to $\mathbf{P}_x$. The process $X$ is assumed to evolve under $\mathbf{P}_x$ and the sample paths are assumed to be right-continuous and left-continuous over stopping times meaning the following: if the sequence of stopping times $\tau_n\uparrow \tau$, then $X_{\tau_n}\rightarrow X_\tau$ $\mathbf{P}_x$-almost surely as $n\rightarrow\infty$. Furthermore, we assume that the underlyings probability space is rich enough to carry a time-shift operator $\theta_\cdot$ for the strong Markov process. There is a well-established theory of standard optimal stopping for this class of processes, see \cite{ps}.

For $r>0$, we denote by $L_1^r$ the class of real valued measurable functions $f$ on $E$ satisfying the integrability condition $\mathbf{E}_x\left\{\int_0^\infty e^{-rt} \left|f(X_t)\right|dt \right\}<\infty$ for all $x\in E$. For a function $f\in L_1^r$, the \emph{resolvent} $R_rf:E\rightarrow\mathbf{R}$ is defined as
\[ (R_rf)(x)=\mathbf{E}_x \left\{\int_0^\infty e^{-rs} f(X_s) ds \right\}, \]
for all $x \in E$. It is well known that the family $(R_\lambda)_{\lambda\geq0}$ is a strongly continuous contraction resolvent and that it has the following connection to exponentially distributed random times: if $U\sim \Exp(\lambda)$ and independent of $X$, then $\lambda(R_{r+\lambda}g)(x)=\mathbf{E}_x[e^{-rU}g(X_U)]$ whenever $g\in L_1^r$, see \cite{Rogers_Williams_I}.

As was explained in the introduction, we consider an optimal multiple stopping problem with underlying strong Markov process $X$ subject to exponentially distributed refraction periods in between the stopping times. To make a precise statement, let $U,U_1,\dots,U_N$ be IID with $U\sim \Exp(\lambda)$ and independent of $X$. For brevity, denote the $N$-tuple $(U_1,\dots,U_N)$ as $\bar{U}$. Furthermore, denote as $\bar{\tau}=(\tau_1,\dots,\tau_N)$ an $N$-tuple of random times in the case  the case $N\in\N$ finite. We first concentrate on the case $N=\infty$, where we consider sequences $\bar{\tau}=(\tau_n)_{n\in\N},\;\bar{U}=(U_n)_{n\in\N},$ instead of tuples. A priori, the random times $U_i$ are not $\mathbb{F}$-stopping times, so we need to consider an appropriately augmented filtration to formalize the multiple stopping problem. To this end, we use the approach used in \cite{CIJ} which we review here shortly for the readers convenience. Define the augmented filtration $\mathbb{F}^{\tau_1+U_1}$ as the smallest right-continuous filtration such that the random time $\tau_1+U_1$ is a stopping time. Furthermore, define the remaining augmentations $\mathbb{F}^{\tau_1+U_1,\dots,\tau_i+U_i}$, $i=2,\dots,N$, recursively as
\[ \mathbb{F}^{\tau_1+U_1,\dots,\tau_i+U_i} = \left(\mathbb{F}^{\tau_1+U_1,\dots,\tau_{i-1}+U_{i-1}}\right)^{\tau_i+U_i}.  \]
This filtration carries the information on the occurrences of the chosen stopping times $\tau_i$ and the random times when the refraction period has elapsed after a given stopping time, i.e., the times $\tau_i+U_i$. Now, for each $\mathbb{F}$-stopping time $\eta$ and for $n\leq N$, define
\begin{displaymath}
\mathcal{S}^n_\eta(\hat{U},\mathbb{F})=\left\{
\overline{\tau} \ : \vphantom{\begin{split} &\tau_1\text{ is a $\mathbb{F}$-stopping time with }\eta\leq\tau_1,\\&\tau_i \text{ is a $\mathbb{F}^{\tau_1+U_1+\dots,\tau_{i-1}+U_{i-1}}$-stopping time } \\& \text{and } \tau{i-1}+U_{i-1}\leq \tau_i, \ i=2,\dots,n \end{split}} \right.
\left.\begin{split} &\tau_1\text{ is a $\mathbb{F}$-stopping time with }\eta\leq\tau_1,\\&\tau_i \text{ is a $\mathbb{F}^{\tau_1+U_1+\dots,\tau_{i-1}+U_{i-1}}$-stopping time } \\& \text{and } \tau_{i-1}+U_{i-1}\leq \tau_i, \ i=2,\dots,n \end{split} \right\},
\end{displaymath}
where $\hat{U}=(U_1\dots,U_{n-1})$. Using the set $\mathcal{S}$, we can formalize the multiple stopping problem under consideration as follows
\begin{equation}\label{Optimal Multiple Stopping}
V^N(x)=V^N_\lambda(x)=\sup_{\bar{\tau}\in \mathcal{S}^N_0(\hat{U},\mathbb{F})} \mathbf{E}_x\left\{ \sum_{i=1}^N e^{-r\tau_i} g(X_{\tau_i}) \mathbf{1}_{\{\tau_i<\infty\}} \right\}.
\end{equation}
Here, $r>0$ is the constant rate of discounting and $g$ has some regularity properties to be specified below.

\subsection{Connection to impulse control}

Before solving the multiple stopping problem, we want to point out the connection of the multiple and infinite optimal stopping problem to impulse control problems. We do not use the theory of impulse control in the following, but this point of view sheds light on the nature of the problem. To this end, we introduce a new Markov process $\hat X$ on the new state space
\[\hat E=E\cup E^\partial,\]
where for each $x\in E$ we introduce a new isolated point $\partial_x$ and write $E^\partial=\{\partial_x:x\in E\}$. Started on $E$ the process $\hat X$ behaves as the original process $X$. Started in $\partial_x$, the process $\hat X$ stays in this point for an exponential time (with parameter $\lambda$) and jumps back to $E$ and is restarted with an initial distribution $\mathbf{P}_x(X_S\in \cdot)$, where $S$ is an independent $\Exp(\lambda)$-distributed random variable.

An admissible impulse control policy for this sequence is a potentially infinite joint sequence $K=(\tau_n,\eta_n)_{n\in\N}$, where $(\tau_n)_{n\in\N}$ is an increasing sequence of stopping times and each $\eta_n$ is an $\mathcal{F}_{\tau_n}$-measurable random variable with values in the set $\mathcal{A}(X_{\tau_n})$ (in case this set is not empty; for $\mathcal{A}(X_{\tau_n})=\emptyset$, no control is possible). In our special situation, we consider $\mathcal{A}(x)=\{\partial_x\}$ for all\ $x\in E$ and $\mathcal{A}(x)=\emptyset$ for all\ $x\in E^\partial$. For each impulse control strategy $K$, the controlled process $X$ under the measure $\mathbf{P}^K_x$ behaves as under $\mathbf{P}_x$ until the first control takes place in $\tau_1$, then the process is started from $\eta_1$ and runs uncontrolled until $\tau_2$ and so on. The construction of the process (on an extension of the original probability space) can be found in \cite{Stettner} to give a reference in English.

We consider the optimal impulse control problem
\[\tilde{V}(x)=\sup_{K=(\tau_n,\eta_n)_{n\in\N}}\mathbf{E}^K_x\left\{ \sum_{n=1}^\infty e^{-r\tau_n} g(X_{\tau_n-}) \mathbf{1}_{\{\tau_n<\infty\}} \right\}\]
where $g:E\rightarrow\mathbf{R}$ is as before and we formally set $g(\partial_x)=0$ for all $x\in E$.

Now, noting  that there is an exponential time between each two stopping times of an impulse control strategy, there is a one-to-one correspondence between sequences of stopping times $\bar{\tau}\in \mathcal{S}^\infty_0(\hat{U},\mathbb{F})$ in the infinite stopping problem and impulse control strategies as given above simply by
\[(\tau_n,\eta_n)_{n\in\N}\mapsto (\tau_n)_{n\in\N}\]
and
\[\mathbf{E}^K_x\left\{ \sum_{n=1}^\infty e^{-r\tau_n} g(X_{\tau_n-}) \mathbf{1}_{\{\tau_n<\infty\}} \right\}=\mathbf{E}_x\left\{ \sum_{i=1}^\infty e^{-r\tau_i} g(X_{\tau_i}) \mathbf{1}_{\{\tau_i<\infty\}}\right\}.\]
In particular, we see that $V^\infty(x)=\tilde{V}(x)$ and $(\tau_n^*,\eta_n^*)_{n\in\N}$ is optimal for the impulse control problem iff $(\tau_n^*)_{n\in\N}$ is optimal for the infinite stopping problem. Therefore, the infinite stopping problem can be seen as an impulse control problem for a strong Markov process from a special class. The multiple stopping problem for finite $N$ can therefore be seen as the standard approximation of this problem, that is often used for an approximation for general impulse control problems. For connected results in the finite stopping situation, we refer to \cite[Subsection 3.6]{C13impulse}.

\section{The infinite stopping problem}\label{sec:infinite}
Since the time horizon is assumed to be infinite, it is natural to consider the infinite stopping problem $N=\infty$, that is
 \begin{equation}\label{Optimal infinite Stopping}
V^\infty(x)=V^\infty_\lambda(x)=\sup_{\bar{\tau}\in \mathcal{S}^\infty_0(\hat{U},\mathbb{F})} \mathbf{E}_x\left\{ \sum_{i=1}^\infty e^{-r\tau_i} g(X_{\tau_i}) \mathbf{1}_{\{\tau_i<\infty\}} \right\}.
\end{equation}
We assume in the following that $g$ is continuous and non-negative. 
The main aim of this section is to show that the solution to this infinite stopping problem can be reduced to the solution of the ordinary optimal stopping problem
 \begin{equation*}
\hat{V}(x)=\sup_{\tau} \mathbf{E}_x\left\{ e^{-(r+\lambda)\tau} g(X_{\tau}) \mathbf{1}_{\{\tau<\infty\}} \right\}
\end{equation*}
under general assumptions (note that the discounting parameter $r+\lambda$ is used instead of $r$). To this end, we first give a verification theorem for the value function.
\begin{proposition} \label{thm:verif_infinite}
Assume that $v$ is nonnegative, $r$-excessive, and
\begin{equation}\label{eq:verif1}v(x)\geq g(x)+\lambda (R_{r+\lambda}v)(x).\end{equation}
Then $V^{\infty}\leq v$.

Furthermore, if $\tau^*$ is a stopping time such that for all $x\in E$
\begin{equation}\label{eq:verif2}\mathbf{E}_x\left\{ e^{-r\tau^*} v(X_{\tau^*}) \mathbf{1}_{\{\tau^*<\infty\}} \right\}=v(x)\end{equation}
and
\begin{align}\label{eq:verif3} &\mathbf{E}_x\left\{ e^{-r\tau^*} v(X_{\tau^*}) \mathbf{1}_{\{\tau^*<\infty\}} \right\}\\&= \mathbf{E}_x\left\{ e^{-r\tau^*} g(X_{\tau^*}) \mathbf{1}_{\{\tau^*<\infty\}} \right\}+\lambda \mathbf{E}_x\left\{ e^{-r\tau^*} (R_{r+\lambda}v)(X_{\tau^*}) \mathbf{1}_{\{\tau^*<\infty\}} \right\},\nonumber\end{align}
then $V^\infty=v$ and the sequence $\bar{\tau}\in \mathcal{S}^\infty_0(\hat{U},\mathbb{F})$ associated to $\tau^*$ given by $\tau_1=\tau^*$,
\[\tau_i=\tau^*\circ\theta_{\tau_{i-1}+U_{i-1}}+\tau_{i-1}+U_{i-1}\mbox{ for all }i\geq 2,\]
is optimal.
\end{proposition}
Before proving the result, let us note that for each nonnegative, $r$-excessive function $v$ it holds that $v\in L^{r+\lambda}_1$ since
\[\mathbf{E}_x\left\{\int_0^\infty e^{-(r+\lambda)t} \left|v(X_t)\right|dt \right\}\leq \int_0^\infty e^{-\lambda t} v(x)dt <\infty,\]
i.e. the function $R_{r+\lambda}v$ is well-defined for all $\lambda$.

\begin{proof}
Let $\bar{\tau}\in \mathcal{S}^\infty_0(\hat{U},\mathbb{F})$ be arbitrary. By noting that $\tau_k\rightarrow\infty$, it is clear that $\E_x\left\{e^{-r\tau_k}\lambda (R_{r+\lambda}v)(X_{\tau_k})\right\}\rightarrow0$. Writing $\tau_{i+1}=\tau_i+\sigma_{i+1}\circ \theta_{\tau_i}$ we obtain
\begin{align*}
\E_x&\left\{ \sum_{i=1}^\infty e^{-r\tau_i} g(X_{\tau_i}) \mathbf{1}_{\{\tau_i<\infty\}} \right\}\\
\leq& \E_x\left\{ \sum_{i=1}^\infty e^{-r\tau_i} (v(X_{\tau_i}) -\lambda (R_{r+\lambda}v)(X_{\tau_i}))\mathbf{1}_{\{\tau_i<\infty\}} \right\}\\
=&\E_x \left\{e^{-r\tau_1} v(X_{\tau_1})\mathbf{1}_{\{\tau_1<\infty\}}\right\}-\lim_{k\rightarrow\infty}\E_x\left\{e^{-r\tau_k}\lambda (R_{r+\lambda}v)(X_{\tau_k})\right\}\\
&+\lim_{k\rightarrow\infty}\sum_{i=1}^k\E_x\left\{  e^{-r\tau_{i+1}} v(X_{\tau_{i+1}})\mathbf{1}_{\{\tau_{i+1}<\infty\}} -e^{-r\tau_i}\lambda (R_{r+\lambda}v)(X_{\tau_i})\mathbf{1}_{\{\tau_i<\infty\}} \right\}\\
=&\E_x \left\{e^{-r\tau_1} v(X_{\tau_1})\mathbf{1}_{\{\tau_1<\infty\}}\right\}\\
&+\sum_{i=1}^\infty\E_x\left\{  e^{-r\tau_{i+1}} v(X_{\tau_{i+1}})\mathbf{1}_{\{\tau_{i+1}<\infty\}} -e^{-r\tau_{i}}\E_{X_{\tau_i}}\left\{e^{-rU_{i}}v(X_{U_{i}})\right\} \mathbf{1}_{\{\tau_i<\infty\}}\right\}\\
=&\E_x \left\{e^{-r\tau_1} v(X_{\tau_1})\mathbf{1}_{\{\tau_1<\infty\}}\right\}\\
&+\sum_{i=1}^\infty\E_x\left\{ e^{-r\tau_{i}} (\E_{X_{\tau_i}}\left\{e^{-r\sigma_{i+1}}v(X_{\sigma_{i+1}})\right\} -\E_{X_{\tau_i}}\left\{e^{-rU_{i}}v(X_{U_{i}})\right\})\mathbf{1}_{\{\tau_i<\infty\}} \right\}\\
\leq& v(x),
\end{align*}
where, in the last step, we used the $r$-excessivity of $v$ together with the fact that $\sigma_{i+1}\geq U_{i}$ for all $i$. This proves the first claim.

For the second claim, note that we have equality in each step for the special sequence.
\end{proof}
Having the previous verification result in mind, we have to find an appropriate candidate function $v$. But on the first view, it is not clear at all how to find such a candidate and for other refraction time distributions, it seems to be impossible to find such a candidate explicitly. But the exponential distribution allows for such a construction.

\subsection{On solving ordinary optimal stopping problems using the resolvent operator}\label{subsec:OS_resolvent}
The key idea is to use a representation of the value function of ordinary optimal stopping problems using the resolvent operator as follows: Consider the ordinary optimal stopping problem
 \begin{equation*}
\hat{V}(x)=\sup_{\tau} \mathbf{E}_x\left\{ e^{-\hat{r}\tau} g(X_{\tau}) \mathbf{1}_{\{\tau<\infty\}} \right\}
\end{equation*}
for some $\hat{r}>0$. It is well-known, that under minimal conditions on $X$ and $g$, the value function $\hat{V}$ can be characterized as the smallest\ $\hat{r}$-excessive majorant of $g$, see e.g. \cite{shiryayev78}. A typical class of $\hat{r}$-excessive functions is given by the $\hat{r}$-resolvent operator $R_{\hat{r}}$ applied to a nonnegative function $\sigma$. On the other hand, the Riesz representation theorem for $\hat{r}$-excessive functions of a \textit{nice} strong Markov process $X$ with state space $E$ yields that each $\hat{r}$-excessive function $u$ can be represented uniquely in the form
\begin{equation}\label{eq:riesz_gen}
u(x)=\int_{E} G_{\hat{r}}(x,y)\, \sigma(dy) + h(x),
\end{equation}
where $G_{\hat{r}}$ denotes the resolvent kernel with respect to some duality measure $m$, $\sigma$ is a Radon measure, and $h$ is an $\hat{r}$-harmonic function. For the exact assumptions on the process $X$ in the framework of Hunt processes, we refer to the discussion in \cite[Chapter 13,14]{CW}. This integral representation of $\hat{r}$-excessive functions can be used fruitfully to establish a dual approach to solving ordinary optimal stopping problems by representing $\hat{V}$ in the form \eqref{eq:riesz_gen}, see \cite{salminen85,MoSa,CI2,CrocceMordecki,Christensen_Salminen}. One basic idea is the following: Under weak integrability assumptions on $g$, it can be seen that $\hat{V}$ is an $\hat{r}$-potential, which yields that $h=0$. Furthermore, if $g$ is smooth enough, it can be seen that -- under some further assumptions -- the measure $\sigma$ is absolutely continuous with respect to the duality measure $m$ and the density of $\sigma$ is given by
\[\sigma(dy)=
\begin{cases}
0,&y\in S^c,\\
\sigma(y)m(dy),&y\in S,
\end{cases}\]
where $S$ denotes the stopping set, $\sigma(y)=(\hat{r}-\mathcal{A})\hat{V}(y)$, and $\mathcal{A}$ denotes the generator/Dynkin operator of $X$; for local operators $\mathcal{A}$, it holds that $\sigma(y)=(\hat{r}-\mathcal{A})g(y)$. We write $\sigma(y)=0$ on $S^c$. Then, the representation \eqref{eq:riesz_gen} reads as
\begin{equation*}
\hat{V}(x)=\int_{E} G_{\hat{r}}(x,y)\, \sigma(dy) + h(x)=\int_E \sigma(y)G_{\hat{r}}(x,y)\, m(dy)=R_{\hat{r}}\sigma(x),
\end{equation*}
which yields a representation of the value function as a resolvent of a nonnegative function $\sigma$. This representation for value functions of ordinary optimal stopping problems will be the key for solving infinite stopping problems in the rest of this section.

\subsection{Reduction of the infinite optimal stopping problem to an ordinary one}
Now, we assume that there exists an optimal stopping time $\tau^*$ for the problem
 \begin{equation}\label{Optimal Stopping Random time horizon}
\hat{V}(x)=\sup_{\tau} \mathbf{E}_x\left\{ e^{-(r+\lambda)\tau} g(X_{\tau}) \mathbf{1}_{\{\tau<\infty\}} \right\}
\end{equation}
and that $\hat{V}$ has the Riesz representation
\begin{equation*}
\hat{V}(x)=(R_{r+\lambda}\sigma)(x)
\end{equation*}
for some function $\sigma:E\rightarrow [0,\infty)$, see the previous discussion. Assuming that the integral defining $\hat{V}$ is finite, we define
\[v(x)=(R_r\sigma)(x)\mbox{ for all }x\in E;\]
note that we changed the parameter of the resolvent from\ $\hat{r}=r+\lambda$ to $r$. Then, $v$ is $r$-excessive and non-negative as the $r$-resolvent of a nonnegative function. Furthermore, since $\hat{V}$ is $r+\lambda$-harmonic on the continuation set $S^c$, i.e. $\sigma=0$ there, we see that $v$ is $r$-harmonic on this set, in particular
\[\mathbf{E}_x\left\{ e^{-r\tau^*} v(X_{\tau^*}) \mathbf{1}_{\{\tau^*<\infty\}} \right\}=v(x).\]
That is, the condition \eqref{eq:verif2} holds. Moreover, using the resolvent equation we obtain
\begin{align*}
v(x)-g(x)\geq v(x)-\hat{V}(x)=((R_r-R_{r+\lambda})\sigma)(x)=\lambda (R_{r+\lambda}(R_r\sigma))(x)=\lambda (R_{r+\lambda}v)(x)
\end{align*}
for all $x\in E$ with equality for $x$ in the stopping set, in particular assumption \eqref{eq:verif1} holds true. By evaluating these functions at $X_{\tau^*}$, multiplying by $e^{-r\tau^*}$ and taking expectations, we obtain using the optimality of $\tau^
*$
\begin{align*}
&\mathbf{E}_x\left\{ e^{-r\tau^*} v(X_{\tau^*}) \mathbf{1}_{\{\tau^*<\infty\}} \right\}-\mathbf{E}_x\left\{ e^{-r\tau^*} g(X_{\tau^*}) \mathbf{1}_{\{\tau^*<\infty\}} \right\}\\
=&\mathbf{E}_x\left\{ e^{-r\tau^*} v(X_{\tau^*}) \mathbf{1}_{\{\tau^*<\infty\}} \right\}-\mathbf{E}_x\left\{ e^{-r\tau^*} \hat{V}(X_{\tau^*}) \mathbf{1}_{\{\tau^*<\infty\}} \right\}\\
=&\lambda \mathbf{E}_x\left\{ e^{-r\tau^*} (R_{r+\lambda}v)(X_{\tau^*}) \mathbf{1}_{\{\tau^*<\infty\}} \right\},
\end{align*}
that is \eqref{eq:verif3}. This shows that that the assumptions of Proposition \ref{thm:verif_infinite} are fulfilled. Putting pieces together, we have the following:
\begin{proposition}\label{prop:infinite stopping}
Assume that there exists an optimal stopping time $\tau^*$ for the problem
 \begin{equation*}
\hat{V}(x)=\sup_{\tau} \mathbf{E}_x\left\{ e^{-(r+\lambda)\tau} g(X_{\tau}) \mathbf{1}_{\{\tau<\infty\}} \right\}
\end{equation*}
and that $\hat{V}$ has the Riesz representation
 \begin{equation}\label{eq:riesz}
 \hat{V}(x)=(R_{r+\lambda}\sigma)(x)\end{equation}
for some function $\sigma:E\rightarrow [0,\infty)$ in $L_1^r$. Then
\[V^\infty(x)=(R_r\sigma)(x)\mbox{ for all }x\in E\]
and the sequence $\bar{\tau}\in \mathcal{S}^\infty_0(\hat{U},\mathbb{F})$ associated to $\tau^*$ given by $\tau_1=\tau^*$,
\[\tau_i=\tau^*\circ\theta_{\tau_{i-1}+U_{i-1}}+\tau_{i-1}+U_{i-1}\mbox{ for all }i\geq 2,\]
is optimal.
\end{proposition}

\begin{remark}
We point out that the optimal stopping problem \eqref{Optimal Stopping Random time horizon} has an interpretation as an optimal stopping problem with random time horizon. To this end, let $U\sim \Exp(\lambda)$ independent of $X$ and $\tau$ be a stopping time. Then
\[ \mathbf{E}_x\left\{ e^{-r\tau} g(X_{\tau})e^{-\lambda\tau}\mathbf{1}_{\{\tau<\infty\}} \right\} = \mathbf{E}_x\left\{ e^{-r\tau} g(X_{\tau})\mathbf{1}_{\{\tau<U\}} \right\}. \]
That is, we can interpret the optimal stopping problem \eqref{Optimal Stopping Random time horizon} as a problem with independent and exponentially distributed time horizon. The mean of the random time horizon is the same as the mean waiting time after stopping in the infinite stopping problem.
\end{remark}

\section{Some examples for the infinite stopping problem}\label{sec:ex_infinite}
\subsection{Infinite American call problem for the geometric Brownian motion}\label{subsec:call_brown}
As an application of the previous results, we consider the case 
that $g(x)=(x-K)^+$ and $X$ is a geometric Brownian motion on $(0,\infty)$, i.e., the regular linear diffusion $X$ given as the solution of the It\^{o} equation $dX_t = \mu X_t dt + \sigma X_t dW_t$,  where $\mu \in \mathbf{R}$ and $\sigma>0$. Here, $W$ is a Wiener process. For the problem to be well-defined, we assume that $\mu<r$. The associated ordinary optimal stopping problem
 \begin{equation}\label{eq:OS_1}
\hat{V}(x)=\sup_{\tau} \mathbf{E}_x\left\{ e^{-(r+\lambda)\tau} (X_{\tau}-K)^+ \mathbf{1}_{\{\tau<\infty\}} \right\}
\end{equation}
has the well-known optimal stopping time
\[\tau^*=\inf\{t\geq0:X_t\geq \hat{x}_\infty\},\;\;\mbox{where}\;\;\hat{x}_\infty=\frac{\beta}{\beta-1}K,\]
for
\[\beta=\left(\frac{1}{2}-\frac{\mu}{\sigma^2} \right)+\sqrt{\left(\frac{1}{2}-\frac{\mu}{\sigma^2} \right)^2+\frac{2(r+\lambda)}{\sigma^2}}>1,\]
and
\begin{equation*}
\hat{V}(x) =
\begin{cases}
x-K, & x \geq \hat{x}_\infty \\
\frac{{\hat{x}_\infty}-K}{{{\hat{x}_\infty}}^\beta} x^\beta, & x \leq {\hat{x}_\infty}.
\end{cases}
\end{equation*}
Writing
\[\sigma(x)=(r+\lambda-\mathcal{A})g(x)=(r+\lambda-\mu)x-K(r+\lambda)\]
for $x\geq {\hat{x}_\infty}$ and $\sigma(x)=0$ for $x<{\hat{x}_\infty}$, we find that $\hat{V}$ can be represented as
\begin{equation*}
\hat{V}(x)=(R_{r+\lambda}\sigma)(x),
\end{equation*}
see also \cite{CrocceMordecki},\ Section 5.1. Therefore, by Proposition \ref{prop:infinite stopping} the sequence $\bar{\tau}\in \mathcal{S}^\infty_0(\hat{U},\mathbb{F})$ associated to $\tau^*$ given by
\[\tau_i=\inf\{t\geq\tau_{i-1}+U_{i-1}:X_t\geq {\hat{x}_\infty}\}\mbox{ for all }i\geq 2,\]
is optimal and
\begin{equation*}
V^\infty(x)=(R_{r}\sigma)(x).
\end{equation*}
Using the representation of the resolvent for diffusion processes discussed in \eqref{Resolvent integral representation} below, a short calculation yields
 \begin{equation*}
V^\infty(x)=\begin{cases}
c_1x+c_2+c_3x^a, & x \geq {\hat{x}_\infty}, \\
c_4x^b, & x < {\hat{x}_\infty},
\end{cases}
\end{equation*}
where
\begin{displaymath}
b=\left(\frac{1}{2}-\frac{\mu}{\sigma^2} \right)+\sqrt{\left(\frac{1}{2}-\frac{\mu}{\sigma^2} \right)^2+\frac{2r}{\sigma^2}}>1, \;
a=\left(\frac{1}{2}-\frac{\mu}{\sigma^2} \right)-\sqrt{\left(\frac{1}{2}-\frac{\mu}{\sigma^2} \right)^2+\frac{2r}{\sigma^2}}<0,
\end{displaymath}
and
\begin{align*}
c_1&=\frac{r+\lambda-\mu}{r-\mu},\;\;
c_2=-K\frac{r+\lambda}{r},\\
c_3&=\frac{2}{\sigma^2}B_r^{-1}{({\hat{x}_\infty})}^{b+2\mu/\sigma^2-1}\left(-\frac{r+\lambda-\mu}{b+2\mu/\sigma^2}{\hat{x}_\infty}-\frac{K(r+\lambda)}{b+2\mu/\sigma^2-1}\right),\\
c_4&=\frac{2}{\sigma^2}B_r^{-1}{({\hat{x}_\infty})}^{a+2\mu/\sigma^2-1}\left(-\frac{r+\lambda-\mu}{a+2\mu/\sigma^2}{\hat{x}_\infty}-\frac{K(r+\lambda)}{a+2\mu/\sigma^2-1}\right).
\end{align*}
Note that in the degenerated case $\lambda=0$, one obtains $V^\infty=\hat{V}$, as expected.

\subsection{Infinite American call problem for geometric L\'evy processes}\label{subsec:call_levy}
One can generalize the previous result to general geometric L\'evy processes, that is $X=e^Y$, where $Y$ is a L\'evy process with $Ee^{Y_1}<e^r$. In this case, it is known that the optimal stopping time for the problem \eqref{eq:OS_1} is also a threshold time and the optimal threshold is given by
\[{\hat{x}_\infty}=KEe^M,\]
where $M$ denotes the running maximum process of $Y$ evaluated at an independent, $\Exp(r+\lambda)$-distributed time, see \cite{M}.

Finding an explicit representation of the form \eqref{eq:riesz} for the value function seems to hard for general L\'evy processes. But for spectrally positive L\'evy processes \cite[Proposition 2.16]{CST} is applicable and yields that again
\[\sigma(y)=(r+\lambda-\mathcal{A})(e^y-K)\;\mbox{ for }y\geq {\hat{x}_\infty},\]
where $\mathcal{A}$ denotes the extended infinitesimal generator of $Y$, which gives the desired representation. In this case, $\mathcal{A}$ acts as
\begin{align*}
\mathcal{A}f(y)=&\frac{c^2}{2}\frac{d^2}{dy^2}f(y)+ b\frac{d}{dy}f(y)\\
&+\int_{(0,\infty)} \left(f(y+z)-f(y)-z\frac{d}{dz}f(z)1_{\{|z|<1\}}\right)\pi(dz),
\end{align*}
where $(b,c,\pi)$ is the L\'evy triple of $X$. Therefore, $\sigma$ can be identified to have the following easy form:
\begin{align}\label{eq:sigma_levy}
\sigma(y)&=(r+\lambda)(e^y-K)+e^y\left(\frac{c^2}{2}+b+\int_{(0,\infty)}({\rm e}^z-1-y1_{|z|<1})\pi(dz)\right)\\
&=-(r+\lambda)K+(r+\lambda+\psi(1)){e}^y,
\end{align}
for $y\geq {\hat{x}_\infty}$ and $=0$ for $y<{\hat{x}_\infty}$, where $\psi(1)=\log \E_0({\rm e}^{Y_1})$. This yields the optimal strategies and the semi-explicit representation for the value function
 \begin{equation*}
V^\infty(x)=(R_{r}\sigma)(x).
\end{equation*}
A more explicit representation of the resolvent operator for spectrally one-sided L\'evy processes is discussed in Subsection \ref{subsec:LP} below.

\subsection{Infinite investment problem}
To illustrate that the theory is not restricted to one-dimensional problems, we now consider one of the most studied multidimensional ordinary optimal stopping problems, namely
\[v^1(x)=\sup_{\tau}\E_{x}(e^{-\beta\tau}(X^{(0)}_\tau-X^{(1)}_\tau-...-X^{(d)}_\tau)), ~~x\in(0,\infty)^{d+1},\beta>0,\]
where $(X^{(0)},...,X^{(d)})$ is a $d+1$-dimensional (correlated) geometric Brownian motion. This problem is motivated by an investment situation: An investor has the possibility to choose a time point to make an investment. She has to pay the sum of the cost factors $X^{(1)},...,X^{(d)}$ and gets out $X^{(0)}$. By a change of measure argument, it can immediately be seen that the problem can be reduced to the following
\begin{equation}\label{eq:value_invest}
\hat{V}(x)=\sup_{\tau} \E_xe^{-\beta\tau}\left(1-\sum_{i=1}^dX_\tau^{(i)}\right)^+,\;\;\;x\in(0,\infty)^d.
\end{equation}
The problem was studied in \cite{DS,OS,ho,NR,CI2,Christensen_Salminen} from different points of view. Recently, the problem was solved \cite{Christensen_Salminen} by characterizing the boundary of the optimal stopping set $S$ as the unique solution to an integral equations using the Riesz representation approach as described in Subsection \ref{subsec:OS_resolvent}. Then, the value function can be represented as
\[\hat{V}(x)=R_\beta\sigma(x)\]
with $\sigma$ given by $\sigma(y)=(\beta-\mathcal{A})g(y)$ on $S$ and $\sigma(y)=0$ on $S^c$. Now, we consider the corresponding infinite stopping problem: Over the time, new investment opportunities arise, but after making an investment, the investor has to wait for the next investment opportunity; as before, we assume this refraction period to be exponentially distributed. Due to the memoryless property of the exponential distribution, this assumption seems to be reasonable in this situation. We are faced with the problem
 \begin{equation*}
V^\infty(x)=\sup_{\bar{\tau}\in \mathcal{S}^\infty_0(\hat{U},\mathbb{F})} \mathbf{E}_x\left\{ \sum_{i=1}^\infty e^{-r\tau_i} g(X_{\tau_i}) \mathbf{1}_{\{\tau_i<\infty\}} \right\},
\end{equation*}
where $X=(X^{(1)},...,X^{(d)})$ is the geometric Brownian motion and $g(x)=(1-x_1-...-x_d)^+$. Again, this problem can immediately be solved by applying Proposition \ref{prop:infinite stopping}: We consider the ordinary optimal investment problem discussed above with parameter $\beta=r+\lambda$ and obtain $S$ and the function $\sigma$. Then
 \begin{equation*}
V^\infty(x)=R_r\sigma(x)
\end{equation*}
and the sequence of optimal stopping times is given recursively by
\[\tau_i=\inf\{t\geq \tau_{i-1}+U_{i-1}:X_t\in S\}.\]
This result can be extended to multidimensional geometric L\'evy processes with only negative jumps by using the results from \cite{Christensen_Salminen}.

\section{The finite stopping problem}\label{sec:mult}
In the previous sections we studied the limiting case of the optimal multiple stopping problem \eqref{Optimal Multiple Stopping} where the number of exercise rights was infinite, i.e $N=\infty$. In this section, we focus on the finite case $N<\infty$. The following theorem gives the solution of the optimal multiple stopping problem \eqref{Optimal Multiple Stopping} for finite $N$ via a strip of $N$ optimal single stopping problems with modified payoff functions. For brevity, denote
\begin{equation}\label{def: payoff H}
H^i(x)=g(x)+\lambda(R_{r+\lambda}V_\lambda^{i-1})(x),
\end{equation}
for all $i=1,\dots,N$, where $V^0_\lambda=0$. We make the following assumptions on the payoff structure:
\begin{itemize}
\item[\textbf{(A1)}] the payoff $g:E\rightarrow[0,\infty)$ is lower semicontinuous and in $L^1_r$,
\item[\textbf{(A2)}] there exist an $r$-harmonic function $h:E\rightarrow\mathbf{R}_+$ such that the function $x\mapsto\frac{g(x)}{h(x)}$ is bounded.
\end{itemize}

We have the following result.

\begin{theorem}\label{thm:main}
Assume that \textbf{(A1)} and \textbf{(A2)} hold. Then, for all $i=1,\dots,N$, the value function $V^i_\lambda$ exists and can be identified recursively as the least $r$-excessive majorant of the function $H^i$. 
Furthermore, if, in addition, the function $g$ is continuous for all $i=1,\dots,N$, then the optimal stopping time $\tau^*_{i}$ exists and can be expressed as
$$\tau^*_{i}=\inf\{ t\geq \tau^*_{i-1}+U_{i-1} \ : H^{N-i+1}(X_t)=V^{N-i+1}_\lambda(X_t) \}.$$
\end{theorem}
\begin{proof} For notational convenience, we concentrate on the case $N=2$; the general case then holds by a straightforward induction.

Fix $\lambda\geq0$ and $x\in E$. Let $\tau$ be an arbitrary $\mathbb{F}$-stopping time. First, we write
\begin{displaymath}
\mathbf{E}_x\left\{ e^{-r\tau} g(X_\tau) \mathbf{1}_{\{\tau<\infty\}} \right\}=\mathbf{E}^h_x\left\{\frac{g(X_\tau)}{h(X_\tau)}\mathbf{1}_{\{\tau<\infty\}}\right\}h(x),
\end{displaymath}
where $\mathbf{E}^h_x$ is the expectation with respect to the probability $\mathbf{P}^h_x$ associated to the Doob's $h$-transform $X^h$, see \cite{Doob}. Since the function $x\mapsto \frac{g(x)}{h(x)}$ is bounded and lower semicontinuous and the resolvent of zero is zero, we conclude that the function
\[ V^1_\lambda(x)=\sup_{\tau_1} \mathbf{E}_x\left\{ e^{-r\tau_1} g(X_{\tau_1}) \mathbf{1}_{\{\tau_1<\infty\}} \right\}  \]
is finite and is the least $r$-excessive majorant of $x\mapsto g(x)$. Since the underlying $X$ is strong Markov, we find that
\[ \mathbf{E}_x\left\{e^{-r(\tau+t)}V^1_\lambda(X_{\tau+t})\right\}=\mathbf{E}_x\left\{\mathbf{E}_x\left\{e^{-r(\tau+t)}V^1_\lambda(X_{\tau+t})\left| \mathcal{F}_\tau\right.\right\}\right\}=\mathbf{E}_x\left\{e^{-r\tau}\mathbf{E}_{X_\tau}\left\{e^{-rt}V^1_\lambda(X_{t})\right\}\right\}, \]
holds for all $t\geq0$. By integrating this expression over the positive reals with respect to $t$ with weight $\lambda e^{-\lambda t}$ we obtain
\begin{align*}
\mathbf{E}_x\left\{e^{-r(\tau+U)}V^1_\lambda(X_{\tau+U})\right\}&=\mathbf{E}_x\left\{e^{-r\tau}\mathbf{E}_{X_\tau}\left\{\lambda \int_0^\infty e^{-(r+\lambda)t}V^1_\lambda(X_{t})dt\right\}\right\}\\&=\mathbf{E}_x\left\{e^{-r\tau} \lambda(R_{r+\lambda}V^1_\lambda)(X_{\tau}) \right\}.
\end{align*}
Now, let $(\tau_1,\tau_2)\in \mathcal{S}^2_0(\hat{U},\mathbb{F})$ be arbitrary. Then we find that
\begin{displaymath}
\begin{split}
\mathbf{E}_x\left\{ e^{-r\tau_1}g(X_{\tau_1})+e^{-r\tau_2}g(X_{\tau_2}) \right\} &\leq \mathbf{E}_x\left\{e^{-r\tau_1}g(X_{\tau_1})+e^{-r(\tau_1+U_1)}V^1_\lambda(X_{\tau_1+U_1}) \right\} \\
&=\mathbf{E}_x\left\{e^{-r\tau_1}(g(X_{\tau_1})+\lambda (R_{r+\lambda}V^1_\lambda)(X_{\tau_1})) \right\},
\end{split}
\end{displaymath}
which, in turn, yields the inequality
\begin{displaymath}
\sup_{(\tau_1,\tau_2)\in \mathcal{S}^2_0(\hat{U},\mathbb{F})}\mathbf{E}_x\left\{ e^{-r\tau_1}g(X_{\tau_1})+e^{-r\tau_2}g(X_{\tau_2}) \right\} \leq \sup_{\tau_1}\mathbf{E}_x\left\{e^{-r\tau_1}(g(X_{\tau_1})+\lambda (R_{r+\lambda}V^1_\lambda)(X_{\tau_1})) \right\}.
\end{displaymath}

To prove the opposite inequality, let $\tau_1$ be arbitrary. Find an optimal stopping sequence such that
\[\mathbf{E}_x\left\{e^{-r\tau_n^*}g(X_{\tau_n^*})\right\}\uparrow \sup_\tau\mathbf{E}_x\left\{e^{-r\tau}g(X_{\tau})\right\}=V^1_\lambda(x)\]
and write
\[\tau_{2,n}:=(\tau_1+U_1)+\tau_n^*\circ\theta_{\tau_1+U_1},\]
where $\theta_\cdot$ denotes the time-shift operator. By Fubini's theorem and the strong Markov property
\begin{align*}
\mathbf{E}_x&\left\{e^{-r\tau_{2,n}}g(X_{\tau_{2,n}})\right\}\\&=\mathbf{E}_x\left\{e^{-r((\tau_1+U_1)+\tau_n^*\circ\theta_{\tau_1+U_1})}g(X_{(\tau_1+U_1)+\tau_n^*\circ\theta_{\tau_1+U_1}})\right\}\\
&=\mathbf{E}_x\left\{e^{-r\tau_1}\mathbf{E}_x\left\{e^{-rU_1}\mathbf{E}_x\left\{e^{-r\tau_n^*\circ\theta_{\tau_1+U_1}}g(X_{(\tau_1+U_1)+\tau_n^*\circ\theta_{\tau_1+U_1}})\Big|\mathcal{F}_{\tau_1+U_1}\right\}\Big|\mathcal{F}_{\tau_1}\right\}\right\}\\
&=\int_0^\infty \mathbf{E}_x\left\{e^{-r\tau_1}\mathbf{E}_x\left\{e^{-rt}\mathbf{E}_x\left\{e^{-r\tau_n^*\circ\theta_{\tau_1+t}}g(X_{(\tau_1+t)+\tau_n^*\circ\theta_{\tau_1+t}})\Big|\mathcal{F}_{\tau_1+t}\right\}\Big|\mathcal{F}_{\tau_1}\right\}\right\}\lambda e^{-\lambda t}dt\\
&=\int_0^\infty \mathbf{E}_x\left\{e^{-r\tau_1}\mathbf{E}_{X_{\tau_1}}\left\{e^{-rt}\mathbf{E}_{X_t}\left\{e^{-r\tau_n^*}g(X_{\tau_n^*})\right\}\right\}\right\}\lambda e^{-\lambda t}dt\\
&=\mathbf{E}_x\left\{e^{-r\tau_1}\mathbf{E}_{X_{\tau_1}}\left\{e^{-rU_1}\mathbf{E}_{X_{U_1}}\left\{e^{-r\tau_n^*}g(X_{\tau_n^*})\right\}\right\}\right\}.
\end{align*}
Since
\[\mathbf{E}_{X_{U_1}}\left\{e^{-r\tau_n^*}g(X_{\tau_n^*})\right\}\uparrow V^1_\lambda(X_{U_1}),\]
we obtain by monotone convergence that
\begin{align*}
\mathbf{E}_x\left\{e^{-r\tau_{2,n}}g(X_{\tau_{2,n}})\right\}&\rightarrow\mathbf{E}_x\left\{e^{-r\tau_1}\mathbf{E}_{X_{\tau_1}}\left\{e^{-rU_1}V^1_\lambda(X_{U_1})\right\}\right\}\\
&=\mathbf{E}_x\left\{e^{-r\tau_1}\lambda R_{r+\lambda}V^1_\lambda(X_{\tau_1})\right\}.
\end{align*}
We have proved
\begin{displaymath}
\sup_{(\tau_1,\tau_2)\in \mathcal{S}^2_0(\hat{U},\mathbb{F})}\mathbf{E}_x\left\{ e^{-r\tau_1}g(X_{\tau_1})+e^{-r\tau_2}g(X_{\tau_2}) \right\} \geq \sup_{\tau_1}\mathbf{E}_x\left\{e^{-r\tau_1}(g(X_{\tau_1})+\lambda (R_{r+\lambda}V^1_\lambda)(X_{\tau_1})) \right\}.
\end{displaymath}
The form of the optimal stopping times holds by the general theory of optimal stopping under the stated assumptions.
\end{proof}

\begin{remark}\label{remark:order}
The previous result can also be obtained by using the general theory developed in \cite{CIJ}.
Furthermore, one sees that the optimal stopping times are first-entrance-times into stopping sets $S^n$, where $S^1$ is the set for one stopping opportunity, $S^2$ for two etc. By the general theory of multiple optimal stopping with general random refraction times, see \cite[Lemma 4.3 ff.]{CIJ}, it is clear that $S^1\subseteq S^2\subseteq \dots \subseteq S^\infty$, where $S^\infty$ is the stopping set for the infinite stopping problem discussed in the previous sections. This fact turns out to be important for solving multiple stopping problems ($N$ finite) in the following.
\end{remark}

\section{Closed-form solutions for the finite stopping problem using the resolvent operator}\label{sec:closed_form} To obtain closed-form solutions of the previous problems using Theorem \ref{thm:main}, it is crucial to have explicit representations for the resolvent of the underlying process. Here, we consider the particularly interesting cases of diffusion processes and (spectrally one-sided) L\'evy processes.

Most solvable ordinary optimal stopping problems have a one-sided solution, that is, the optimal stopping time is of threshold-type. Therefore, it is our aim in the section to find sufficient conditions that guarantee that the optimal stopping times in the multiple stopping problem are also of this type. We give conditions that can be checked a priori, i.e. without solving the sequence of stopping problems using Theorem \ref{thm:main} explicitly.

\subsection{Diffusion dynamics}\label{subsec:diff_one_sided} We assume that the state process $X$ evolves on $\mathbf{R}_+$ and follows the regular linear diffusion given as the weakly unique solution of the It\^{o} equation
\begin{equation}\label{SDE}
dX_t = \mu(X_t)dt+\sigma(X_t)dW_t, \ X_0=x.
\end{equation}
Here, $W$ is a Wiener process on $(\Omega,\mathcal{F},\mathbb{F},\mathbf{P})$ and the real valued functions $\mu$ and $\sigma>0$ are assumed to be continuous. Using the terminology of \cite{BS}, the boundaries $0$ and $\infty$ are natural, see \cite{BS}, pp. 18--20, for the boundary classification of diffusions. As usually, we denote as $\mathcal{A}=\frac{1}{2}\sigma^2(x)\frac{d^2}{dx^2}+\mu(x)\frac{d}{dx}$ the second order linear differential operator associated to $X$. Furthermore, we denote as, respectively, $\psi_r>0$ and $\varphi_r>0$ the increasing and decreasing solution of the ODE $\mathcal{A}u=ru$, where $r>0$, defined on the domain of the characteristic operator of $X$. By posing appropriate boundary conditions depending on the boundary classification of the diffusion $X$, the functions $\psi_r$ and $\varphi_r$ are defined uniquely up to a multiplicative constant and can be identified as the minimal $r$-\replaced{harmonic}{excessive} functions -- for the boundary conditions and further properties of $\psi_r$ and $\varphi_r$, see \cite{BS}, pp. 18--20. Finally, we define the speed measure $m$ and the scale function $S$ of $X$ via the formul\ae~$m'(x)=\frac{2}{\sigma^2(x)}e^{B(x)}$ and  $S'(x)= e^{-B(x)}$ for all $x \in \mathbf{R}_+$, where $B(x):=\int^x \frac{2\mu(y)}{\sigma^2(y)}dy$, see \cite{BS}, pp. 17.

We know from the literature that for a given $f\in L_1^r$ the resolvent $R_rf$ can be expressed as
\begin{equation}\label{Resolvent integral representation}
(R_rf)(x)=B_r^{-1}\varphi_r(x)\int_0^x \psi_r(y)f(y)m'(y)dy+B_r^{-1}\psi_r(x)\int_x^\infty \varphi_r(y)f(y)m'(y)dy,
\end{equation}
for all $x \in \mathbf{R}_+$, where $B_r=\frac{\psi_r'(x)}{S'(x)}\varphi_r(x)-\frac{\varphi_r'(x)}{S'(x)}\psi_r(x)$ denotes the Wronskian determinant, see \cite{BS}, pp. 19. Finally, we remark that the value of $B_r$ does not depend on the state variable $x$ but depends on the rate $r$.

Given the underlying $X$, we consider the multiple optimal stopping problem
\begin{equation}\label{Optimal Multiple Stopping Diffusion}
V^N_\lambda(x)=\sup_{\bar{\tau}\in \mathcal{S}^N_0(\hat{U},\mathbb{F})} \mathbf{E}_x\left\{ \sum_{i=1}^N e^{-r\tau_i} g(X_{\tau_i}) \mathbf{1}_{\{\tau_i<\infty\}} \right\},
\end{equation}
with payoff function $g$, we make more specific assumptions on $g$ later.

For some preliminary analysis, define the operator $L$ as follows
\begin{equation}\label{def: operator L}
(Lf)(x)=\frac{\psi'_r(x)}{S'(x)}f(x)-\frac{f'(x)}{S'(x)}\psi_r(x).
\end{equation}
Then we have the following lemma.
\begin{lemma}\label{Operator L Lemma}
Assume that the function $f$ is continuous and $r$-excessive for diffusion $X$. Then the function $x\mapsto\frac{\lambda(R_{r+\lambda}f)(x)}{\psi_r(x)}$ is decreasing.
\end{lemma}

\begin{proof}
Since the function $\psi_r$ is $r$-harmonic, a straightforward differentiation yields
\begin{align*}
(LR_{r+\lambda}f)'(x)&=-\psi_r(x)(\mathcal{A}-r)(R_{r+\lambda}f)(x)m'(x)
\\&=-\psi_r(x)(\lambda(R_{r+\lambda}f)(x)-f(x))m'(x)\geq 0,
\end{align*}
for all $x\in\mathbf{R}_+$. Here, the operator $L$ is defined in \eqref{def: operator L} and the last inequality is given by Prop. II.2.3 in \cite{BG}. Thus
\begin{align*}
\frac{d}{dx}\left(\frac{\lambda(R_{r+\lambda}f)(x)}{\psi_r(x)}\right)&=-\frac{S'(x)}{\psi_r^2(x)} \lambda(L (R_{r+\lambda}f))(x) \\&= \frac{S'(x)}{\psi_r^2(x)} \lambda \int_0^x \psi_r(y)(\lambda(R_{r+\lambda}f)(y)-f(y))m'(y)dy\leq 0.
\end{align*}
for all $x\in\mathbf{R}_+$.
\end{proof}

In what follows, we pin down sufficient conditions for the optimal stopping rules to be one-sided threshold rules.

\begin{proposition}\label{prop:diffusion}
Assume that
\begin{itemize}
\item the function $g \in C^2(\mathbf{R}_+ \setminus D)\cap C(\mathbf{R}_+)$ such that limits $lim_{x\rightarrow y\pm}g'(x)$ and $lim_{x\rightarrow y\pm}g''(x)$ are finite for all $y\in D$. Here, $D$ is a countable subset of $\mathbf{R}_+$ which has no accumulation points,
\item the function $x\mapsto\frac{g(x)}{\psi_r(x)}$ has a unique finite global maximum at $\hat{x}$ and is decreasing for all $x>\hat{x}$,
\item the function $(\mathcal{A}-r)g$ is decreasing for $x> \hat{x}_\infty$, where $\hat{x}_\infty:=\inf S^\infty$, see Remark \ref{remark:order}.
\end{itemize}
Then, for all $i=1,\dots,N$
\begin{itemize}
\item the function $x\mapsto\frac{H^i(x)}{\psi_r(x)}$ has a finite global maximum at a point $x^*_i\leq\hat{x}$,
\item $S_i=[x_i^*,\infty)$ and the optimal stopping times $\tau_1,\dots,\tau_N$ read as
\[ \tau^*_i = \inf\{ t \geq \tau^*_{i-1}+U_{i-1} : X_t\geq x^*_{N-i+1} \} \]
\item the value functions read as
\begin{equation}
V^{i}_\lambda(x) =
\begin{cases}
H^{i}(x), & x \geq x^*_{{i}} \\
\frac{H^{i}(x^*_{i})}{\psi_r({x^*_{i}})} \psi_r(x), & x \leq x^*_{i},
\end{cases}
\end{equation}
\end{itemize}
where the function $H^i$ is defined in \eqref{def: payoff H}.
\end{proposition}

\begin{proof} Since $V_\lambda^0=0$, the claim follows immediately for $i=1$ by \cite[Theorem 3]{Al01}. We proceed by induction: assume that the claim holds for index $i-1$.
\begin{enumerate}
\item
We know from the standard theory of optimal stopping, see, e.g., \cite{ps}, that the value function $V_{\lambda}^{i-1}$ is finite and $r$-excessive. By Prop. II.2.3 in \cite{BG}, this implies that $V_\lambda^{i-1}(x)\geq\lambda(R_{r+\lambda}V_\lambda^{i-1})(x)$ for all $x\in\mathbf{R}_+$. Thus
\begin{align}\label{H^i computation}
\frac{H^i(x)}{\psi_r(x)}\leq \frac{g(x)+V_\lambda^{i-1}(x)}{\psi_r(x)} \leq \frac{g(\hat{x})}{\psi_r(\hat{x})}+\frac{H^{i-1}(x^*_{i-1})}{\psi_r(x^*_{i-1})},
\end{align}
for all $x\in\mathbf{R}_+$, that is, the function $x\mapsto\frac{H^i(x)}{\psi_r(x)}$ is bounded. In addition, we know from Lemma \ref{Operator L Lemma} that the function $x\mapsto\frac{\lambda(R_{r+\lambda}V_\lambda^{i-1})(x)}{\psi_r(x)}$ is decreasing. Thus the function $x\mapsto\frac{H^i(x)}{\psi_r(x)}$ is decreasing for all $x>\hat{x}$.
Since the interval $(0,\hat{x}_\infty)$ is a subset of the continuation region for all $i=1,\dots,N$ (see Remark \ref{remark:order}), we can assume, without loss of generality, that $g(x)=0$ for all $x\in(0,\hat{x}_\infty)$. Thus we observe using \eqref{H^i computation} that $\frac{H^i(x)}{\psi_r(x)}\leq\frac{H^{i-1}(x^*_{i-1})}{\psi_r(x^*_{i-1})}\leq\frac{H^i(x^*_{i-1})}{\psi_r(x^*_{i-1})}$ for all $x\in(0,\hat{x}_\infty)$. By continuity, we conclude that the function $x\mapsto\frac{H^i(x)}{\psi_r(x)}$ has at least one maximum point on $(\hat{x}_\infty,\hat{x})$. Let $x_i^*$ denote the smallest maximum point.
\item By \cite[Lemma 1.1]{CI2}, we know that $x_i^*\in S^i$. By \cite[Theorem 2]{BL00} we furthermore know that $\inf\{t\geq0:X_t=x_i^*\}$ is an optimal stopping time for all starting points in $(0,x_i^*)$ and by the minimality of\ $x_i^*$ we obtain that $S^i\subseteq [x_{i}^*,\infty)$. On the other hand, by Remark
\ref{remark:order} it holds that $[x_{i-1}^*,\infty)=S^{i-1}\subseteq S^i$.

\item We now show that $x\mapsto\frac{H^i(x)}{\psi_r(x)}$ is non-increasing on $[x_i^*,x_{i-1}^*]$. To this end, consider the function
\begin{equation*}
I^i(x)=\frac{\psi_r'(x)}{S'(x)}H^i(x) - \frac{{H^i}'(x)}{S'(x)}\psi_r(x).
\end{equation*}
Since $\psi_r$ is $r$-harmonic, straightforward differentiation yields
\begin{equation*}
{I^i}'(x)=-\psi_r(x)(\mathcal{A}-r)H^i(x)m'(x),
\end{equation*}
for all $x\in\mathbf{R}_+$. Since the boundaries are natural, we find that $I^i(0)=0$ and, consequently, that
\begin{equation*}
{I^i}(x)=-\int_0^x \psi_r(y)(\mathcal{A}-r)H^i(y)m'(y) dy.
\end{equation*}
On the other hand, since
\begin{equation*}
I^i(x)=-\frac{\psi^2_r(x)}{S'(x)}\frac{d}{dx}\left( \frac{H^i(x)}{\psi_r(x)} \right),
\end{equation*}
we conclude
\begin{equation}\label{Diffusion proof H/psi derivative}
\frac{d}{dx}\left( \frac{H^i(x)}{\psi_r(x)} \right) = \frac{S'(x)}{\psi^2_r(x)}\int_0^x \psi_r(y)(\mathcal{A}-r)H^i(y)m'(y) dy.
\end{equation}
We know that the resolvent $(R_{r+\lambda}V^{i-1}_\lambda)$ satisfies the relation $(\mathcal{A}-(r+\lambda))(R_{r+\lambda}V^{i-1}_\lambda)=-V^{i-1}_\lambda$. Thus
\begin{equation}\label{Diffusion proof A-r H}
(\mathcal{A}-r)H^i(x) = (\mathcal{A}-r)g(x)+\lambda(\lambda(R_{r+\lambda}V^{i-1}_\lambda)(x)-V^{i-1}_\lambda(x)).
\end{equation}
Since the boundary $\infty$ is natural, it follows from Lemma 2.1 in \cite{Lempa02} that
\begin{equation*}
V^{i-1}(x)=\frac{H^{i-1}(x^*_{i-1})}{\psi_r(x^*_{i-1})}\psi_r(x)= \lambda\left( R_{r+\lambda} \frac{H^{i-1}(x^*_{i-1})}{\psi_r(x^*_{i-1})}\psi_r \right)(x)
\end{equation*}
for all $x\leq x^*_{i-1}$. By invoking the representation \eqref{Resolvent integral representation}, we find that
\begin{align*}
\lambda&(R_{r+\lambda}V^{i-1}_\lambda)'(x)-{V^{i-1}_\lambda}'(x)\\&=\frac{\lambda}{B_{r+\lambda}}\left(\varphi_{r+\lambda}'(x)\int_0^x \psi_{r+\lambda}(y)V^{i-1}_\lambda(y)m'(y)dy+\psi_{r+\lambda}'(x)\int_x^\infty \varphi_{r+\lambda}(y)V^{i-1}_\lambda(y)m'(y)dy \right) \\&-
\frac{\lambda}{B_{r+\lambda}}\frac{H^{i-1}(x^*_{i-1})}{\psi_r(x^*_{i-1})}\left(\varphi_{r+\lambda}'(x)\int_0^x \psi_{r+\lambda}(y)\psi_r(y)m'(y)dy+\psi_{r+\lambda}'(x)\int_x^\infty \varphi_{r+\lambda}(y)\psi_r(y)m'(y)dy \right) \\& \leq 0.
\end{align*}
Thus, we find using \eqref{Diffusion proof A-r H} that the integrand in \eqref{Diffusion proof H/psi derivative} is non-increasing on $(\hat{x}_\infty,x_{i-1}^*)$. Therefore the function is $x\mapsto\frac{H^i(x)}{\psi_r(x)}$ is non-increasing on $[x_i^*,x_{i-1}^*]$.
\item To prove that $S^i=[x_i^*,\infty)$ it remains to be shown that $[x_i^*,x_{i-1}^*]\subseteq S^i$. To this end, assume that there exist $y_1<y_2\in[x^*_i,x^*_{i-1}]$ such that $(y_1,y_2)\not\in S^i$. By  \cite[Theorem 3]{BL00} (or also \cite{CI2}), there exists $\lambda\in(0,1)$ such that $y_1,y_2$ are maximum points of $\frac{H^i}{\lambda \varphi_r+(1-\lambda)\psi_r}$. Without loss of generality, we standardize the functions $\psi_r$ $\varphi_r$ such that $\varphi_r(y_1)=\psi_R(y_1)$. Then - using that $H^i/\psi_r$ is non-increasing on $[x^*_{i-1},x^*_i]$ and $\varphi_r(y_2)<\psi_r(y_2)$ - we obtain
\begin{align*}
\frac{H^i}{\lambda \varphi_r+(1-\lambda)\psi_r}(y_1)&=\frac{H^i}{\varphi_r}(y_1)\geq\frac{H^i}{\varphi_r}(y_2)>\frac{H^i}{\lambda \varphi_r+(1-\lambda)\psi_r}(y_2),
\end{align*}
which is a contradiction.
\end{enumerate}
Now, all other claims hold by the general theory of optimal stopping.
\end{proof}

Proposition \ref{prop:diffusion} 
is formulated for increasing payoff satisfying additional regularity conditions. We point out that one can formulate an analogous result also for decreasing payoffs in terms of the ratio function $x\mapsto \frac{g(x)}{\varphi_r(x)}$ and obtain another wide class of solvable optimal multiple stopping problems with one-sided optimal threshold rules -- for analogous results in optimal single stopping, see, e.g., \cite{Lempa} and \cite{Lempa02}.

\subsection{L\'evy processes}\label{subsec:LP}
For L\'evy processes $X$, we consider the resolvent kernel $G_r$ given by
\[G_r(x,y)dy=\int_0^\infty e^{-rt}\mathbf{P}_x(X_t\in dy)dt.\]
Assuming that $M$ and $I$ are independent random variables with distributions of $\sup_{t\leq T}X_t$ and $\inf_{t\leq T}X_t$, respectively, where $T$ is an independent exponential time with parameter $r$, the Wiener-Hopf-factorization (see \cite[Theorem 6.16]{Kyp}) states that
\[X_T\stackrel{d}{=}M+I.\]
Therefore,
\[rG_r(x,y)=\begin{cases}
\int_{-\infty}^{y-x}f_I(t)f_M(y-x-t)dt,&y<x\\
\int^{\infty}_{y-x}f_M(t)f_I(y-x-t)dt,&y>x,
\end{cases}\]
where $f_I$, $f_M$ denote the densities of $I$ and $M$, that we assume to exist.\\
Unfortunately, these densities are often not easy to find explicitly, so that we consider the particularly interesting class of spectrally negative L\'evy processes now. In this case, the resolvent kernel $G_r$ can be given semi-explicitly in terms of the scale-function $W^{(r)}$ of the process and the right inverse $\Phi$ of the Laplace exponent as
 \begin{align}\label{eq:green_levy}
 G_r(x,y)=\Phi'(r)e^{-\Phi(r)(y-x)}-W^{(r)}(x-y)
 \end{align}
and $x^*$, $\sigma$ are given as above, see \cite[Corollary 8.9]{Kyp}. The resolvent is then given by
\[ (R_rf)(x)=\int_{\mathbb{R}}f(y)G_r(x,y)dy.\]
This representation can be used as a key for solving multiple stopping problems for underlying L\'evy processes in our setting.
Note that for each spectrally positive L\'evy process $X$, the process $-X$ is a spectrally negative L\'evy process, so that the analogous formulas can be used for spectrally positive L\'evy processes also.

To give a basis for obtaining explicit examples, we again concentrate on examples that lead to optimal stopping times of threshold-type. As a main tool, we use the theory developed in \cite{CST}.

\begin{proposition}\label{prop:levy}
Let $X$ be a spectrally positive L\'evy process and let $g$ be such that there exists a continuous function $\tilde{f}$ such that
\begin{itemize}
\item $g(x)=R_r\tilde{f}(x)$ for all $x$,\footnote{under appropriate smoothness assumptions on $g$, this means that $\tilde{f}=(r-\mathcal{A})g$}
\item the function $\tilde{f}$ is non-decreasing. 
\end{itemize}
Then for each $N\in\N$, there exists $\hat{x}_N\geq \hat{x}_\infty$ such that $S^N=[x_N^*,\infty)$ and the value function has the form
$V^N=R_r\sigma^N$ for some continuous non-decresing function $\sigma^N$, which fulfills the recursive equation
\[\sigma^N=\tilde{f}+\lambda R_{r+\lambda}\sigma^{N-1}\mbox{ on }[x_N^*,\infty)\mbox{ and }\sigma^N=0\mbox{ on }(-\infty,x_N^*].\]
\end{proposition}

\begin{proof}
We proceed by induction on $N$ and assume that
\[S^{N-1}=[x^*_{N-1},\infty),\;\;V^{N-1}=R_r\sigma^{N-1},\]
where $\sigma^{N-1}$ is continuous and non-decreasing. Note that $\sigma^{N-1}(x)=0$ for $x\leq x^*_{N-1}$. By Theorem \ref{thm:main}, we are faced with the optimal stopping problem with reward function
\[g+\lambda R_{r+\lambda}V^{N-1}=R_r\tilde{f}+ \lambda R_{r+\lambda}R_r\sigma^{N-1}=R_r\tilde{f}^N,\]
where $\tilde{f}^N=\tilde{f}+\lambda R_{r+\lambda}\sigma^{N-1}$. $\tilde{f}^N$ is also continuous and non-decreasing. Writing
\[\hat{f}(z)=\frac{1}{r}\int_{-\infty}^0\tilde{f}^N(z+y)\mathbb{P}_0(I_T\in dy),\]
we immediately see that $\hat{f}$ fulfills the assumptions of \cite[Theorem 2.5]{CST}, i.e. $S^N$ is of the form $[x_N^*,\infty)$ for some $x_N^*$, and \cite[Proposition 2.16]{CST} yields the desired resolvent-representation for $V^N$.
\end{proof}

\begin{remark}
The arguments in the previous proof were not directly based on the special structure of the process $X$. The arguments can be generalized to more general real-valued Hunt processes with only positive jumps under appropriate additional assumptions following the results in \cite{CST} carefully.
\end{remark}

\section{An Illustration of the finite stopping problem}\label{sec:call}

To illustrate our results on the finite stopping problem, we come back to the problem discussed in Subsection \ref{subsec:call_brown} and \ref{subsec:call_levy} for $N=\infty$. For general underlying geometric spectrally positive L\'evy processes, it is clear from the discussion in Subsection \ref{subsec:call_levy} that
\[e^x-K=R_r\tilde{f}(x),\]
with $\tilde{f}(x)=c_1e^x-c_2,$ where $c_1,c_2$ are given in \eqref{eq:sigma_levy}. From this explicit representation, it can immediately be seen that the assumptions of Proposition \ref{prop:levy} are fulfilled and that the optimal stopping times are one-sided. To obtain more explicit results, we now consider the case where the underlying process $X$ follows a geometric Brownian motion and apply the results described in Subsection \ref{subsec:diff_one_sided}. The scale density $S'$ reads as $S'(x)=x^{-\frac{2\mu}{\sigma^2}}$ and the speed density $m'$ reads as $m'(x)=\frac{2}{(\sigma x)^2} x^{\frac{2\mu}{\sigma^2}}$. It is well known that the differential operator $\mathcal{A}=\frac{1}{2}\sigma^2x^2\frac{d^2}{dx^2}+\mu x \frac{d}{dx}$. For the sake of finiteness, we assume that $\mu<r$ and $\mu-\frac{1}{2}\sigma^2>0$. This guarantees that the optimal exercise thresholds are finite and are attained almost surely in finite time. The minimal excessive functions $\psi_\cdot$ and $\varphi_\cdot$ can be written as
\begin{displaymath}
\begin{cases}
\psi_r(x)=x^b, \\ \varphi_r(x)=x^a,
\end{cases}
\begin{cases}
\psi_{r+\lambda}(x)=x^\beta, \\ \varphi_{r+\lambda}(x)=x^\alpha,
\end{cases}
\end{displaymath}
where the constants
\begin{displaymath}
\begin{cases}
b=\left(\frac{1}{2}-\frac{\mu}{\sigma^2} \right)+\sqrt{\left(\frac{1}{2}-\frac{\mu}{\sigma^2} \right)^2+\frac{2r}{\sigma^2}}>1, \\
a=\left(\frac{1}{2}-\frac{\mu}{\sigma^2} \right)-\sqrt{\left(\frac{1}{2}-\frac{\mu}{\sigma^2} \right)^2+\frac{2r}{\sigma^2}}<0, \\
\end{cases}
\begin{cases}
\beta=\left(\frac{1}{2}-\frac{\mu}{\sigma^2} \right)+\sqrt{\left(\frac{1}{2}-\frac{\mu}{\sigma^2} \right)^2+\frac{2(r+\lambda)}{\sigma^2}}>1, \\
\alpha=\left(\frac{1}{2}-\frac{\mu}{\sigma^2} \right)-\sqrt{\left(\frac{1}{2}-\frac{\mu}{\sigma^2} \right)^2+\frac{2(r+\lambda)}{\sigma^2}}<0. \\
\end{cases}
\end{displaymath}
It is a simple computation to show that the Wronskian $B_{r+\lambda}=\frac{}{}2\sqrt{\left(\frac{1}{2}-\frac{\mu}{\sigma^2} \right)^2+\frac{2(r+\lambda)}{\sigma^2}}$.

Assume that the exercise payoff reads as $g(x)=(x-K)^+$, where $K$ is a fixed strike price. Furthermore, fix the parameter $\lambda>0$. Then the optimal multiple stopping problem reads as
\[ V^N_\lambda(x)=\sup_{\bar{\tau}} \mathbf{E}\left\{ \sum_{i=1}^N e^{-r\tau_i} (X_{\tau_i}-K)^+ \right\}, \]
for all $i=1,\dots,N$. It is known from the literature that for $N=1$, the optimal stopping threshold and the value read as
\begin{equation*}
x_1^*=\frac{b}{b-1}K, \quad V^1_\lambda(x) =
\begin{cases}
x-K, & x \geq x^*_1 \\
\frac{x^*_1-K}{{x^*_1}^b} x^b, & x \leq x^*_1.
\end{cases}
\end{equation*}

To characterize the thresholds for $i=2,\dots,N$, the objective is to find the state
\[ x^*_{i} =\argmax \left( x\mapsto \frac{H^i(x)}{x^b} \right). \]
This yields the necessary condition
\begin{equation}\label{GBM: Necessary condition}
\left[-\frac{d}{dx}\left(\frac{\lambda(R_{r+\lambda}V^{i-1}_\lambda)(x)}{x^b}\right)\right]_{x=x^*_{i}} = \left[\frac{d}{dx}\left(\frac{(x-K)^+}{x^b}\right)\right]_{x=x^*_{i}}.
\end{equation}
Inductively, we assume that
\begin{equation*}
x^*_{i-1} =\argmax \left( x\mapsto \frac{H^{i-1}(x)}{x^b} \right), \quad V^{i-1}_\lambda(x) =
\begin{cases}
H^{i-1}(x), & x \geq x^*_{i-1} \\
\frac{H^{i-1}(x^*_{i-1})}{{x^*_{i-1}}^b} x^b, & x \leq x^*_{i-1}.
\end{cases}
\end{equation*}
For brevity, denote $c^*_{i-1}=\frac{H^{i-1}(x^*_{i-1})}{{x^*_{i-1}}^b}$. Let $x\leq x^*_{i-1}$. Then we find that
\begin{align*}
\lambda(R_{r+\lambda}V^{i-1}_\lambda)(x)=
\frac{2\lambda}{\sigma^2B_{r+\lambda}}&\left[ x^\alpha \int_0^x y^\beta c^*_{i-1} y^b y^{\frac{2\mu}{\sigma^2}-2}dy+x^\beta \int_x^{x^*_{i-1}} y^\alpha c^*_{i-1} y^b y^{\frac{2\mu}{\sigma^2}-2}dy \right.
\\  &\left.+ x^\beta \int_{x^*_{i-1}}^\infty y^\alpha H^{i-1}(y) y^{\frac{2\mu}{\sigma^2}-2}dy   \right].
\end{align*}
Integration by parts yields
\begin{align*}
\lambda(R_{r+\lambda}V^{i-1}_\lambda)(x)=
\frac{2\lambda}{\sigma^2B_{r+\lambda}}& \left[ c^*_{i-1}\frac{\kappa+\gamma}{\kappa\gamma}x^b + x^\beta \frac{1}{\kappa}\int_{x^*_{i-1}}^\infty y^{-\kappa} \frac{d}{dy}\left( \frac{H^{i-1}(y)}{y^b} \right) dy \right]\\
= V^{i-1}_\lambda(x)&+x^\beta \frac{\gamma}{\gamma+\kappa}\int_{x^*_{i-1}}^\infty y^{-\kappa} \frac{d}{dy}\left( \frac{H^{i-1}(y)}{y^b} \right) dy,
\end{align*}
where
\begin{equation*}
\begin{cases}
\kappa=\sqrt{\left(\frac{1}{2}-\frac{\mu}{\sigma^2} \right)^2+\frac{2(r+\lambda)}{\sigma^2}}-\sqrt{\left(\frac{1}{2}-\frac{\mu}{\sigma^2} \right)^2+\frac{2r}{\sigma^2}}, \\
\gamma=\sqrt{\left(\frac{1}{2}-\frac{\mu}{\sigma^2} \right)^2+\frac{2(r+\lambda)}{\sigma^2}}+\sqrt{\left(\frac{1}{2}-\frac{\mu}{\sigma^2} \right)^2+\frac{2r}{\sigma^2}}.
\end{cases}
\end{equation*}
On the other hand, let $x\geq x^*_{i-1}$. Then
\begin{align*}
\lambda(R_{r+\lambda}V^{i-1}_\lambda)(x)=
\frac{2\lambda}{\sigma^2B_{r+\lambda}}&\left[ x^\alpha \int_0^{x^*_{i-1}} y^\beta c^*_{i-1} y^b y^{\frac{2\mu}{\sigma^2}-2}dy+  x^\alpha \int_{x^*_{i-1}}^x y^\beta H^{i-1}(y) y^{\frac{2\mu}{\sigma^2}-2}dy \right.
\\  &\left.+  x^\beta \int_{x}^\infty y^\alpha H^{i-1}(y) y^{\frac{2\mu}{\sigma^2}-2}dy \right].
\end{align*}
Again, integration by parts yields
\begin{align*}
\lambda(R_{r+\lambda}V^{i-1}_\lambda)(x)
=\frac{2\lambda}{\sigma^2B_{r+\lambda}}&\left[ x^\alpha c^*_{i-1}\frac{1}{\gamma}{x^*_{i-1}}^\gamma +x^\alpha\int_{x^*_{i-1}}^x y^{\gamma-1}\frac{H^{i-1}(y)}{y^b} dy  \right.
\\& \left.+x^{\beta} \int_x^\infty y^{-\kappa-1}\frac{H^{i-1}(y)}{y^b} dy  \right].\\
=\frac{2\lambda}{\sigma^2B_{r+\lambda}}&\left[ x^\alpha c^*_{i-1}\frac{1}{\gamma}{x^*_{i-1}}^\gamma  + \frac{1}{\gamma}x^\alpha\left(x^\gamma\frac{H^{i-1}(x)}{x^b}-{x^*_{i-1}}^\gamma\frac{H^{i-1}({x^*_{i-1}})}{{x^*_{i-1}}^b}\right) \right. \\
& \left.- \frac{1}{\gamma}x^\alpha\int_{x^*_{i-1}}^x y^\gamma\frac{d}{dy}\left(\frac{H^{i-1}(y)}{y^b}\right)dy\right.\\
&+\left.\frac{1}{\kappa}x^\beta\left( x^{-\kappa}\frac{H^{i-1}(x)}{x^b} + \int_x^\infty y^{-\kappa}\frac{d}{dy}\left(\frac{H^{i-1}(y)}{y^b}\right)dy \right)\right]\\
=H^{i-1}(x)+\frac{2\lambda}{\sigma^2B_{r+\lambda}}&\left[\frac{1}{\kappa}x^\beta\int_x^\infty y^{-\kappa}\frac{d}{dy}\left(\frac{H^{N-1}(y)}{y^b}\right)dy\right. \\ &- \left.\frac{1}{\gamma}x^\alpha\int_{x^*_{i-1}}^x y^\gamma\frac{d}{dy}\left(\frac{H^{N-1}(y)}{y^b}\right)dy\right].
\end{align*}
Summarizing, a round of differentiation yields
\begin{align}\label{GBM: resolvent reduction}
\frac{d}{dx}\left(\frac{\lambda(R_{r+\lambda}V^{i-1}_\lambda)(x)}{x^b}\right)=\frac{\kappa\gamma}{\kappa+\gamma}&\left(x^{\kappa-1}\int_{\max\{x, x^*_{i-1}\}}^\infty y^{-\kappa} \frac{d}{dy}\left( \frac{H^{i-1}(y)}{y^b} \right) dy\right. \\
+& \nonumber \left.x^{-\gamma-1}\int_{x^*_{i-1}}^{\max\{x,x^*_{i-1}\}} y^{\gamma}\frac{d}{dy}\left( \frac{H^{i-1}(y)}{y^b} \right) dy\right),
\end{align}
for all $x\in\mathbf{R}^+$. For brevity, denote
\[ \Delta_{i-1} = \frac{\kappa\gamma}{\kappa+\gamma}\int_{x^*_{i-1}}^\infty y^{-\kappa} \frac{d}{dy}\left( \frac{H^{i-1}(y)}{y^b} \right) dy. \]
Then the necessary condition \eqref{GBM: Necessary condition} can be expressed as
\begin{equation*}
{x^*_{i}}-b({x^*_{i}}-K)=-\Delta_{i-1} {x^*_{i}}^\beta.
\end{equation*}
By further simplification we obtain
\begin{equation*}
1-\frac{{x^*_{i}}}{x^*_1} = -\frac{\Delta_{i-1} {x^*_1}^{\beta-1}}{b-1}\left(\frac{{x^*_{i}}}{x^*_1}\right)^\beta. 
\end{equation*}
Since $\beta>1$ and $\Delta_i<0$, we observe that this necessary condition has a unique solution $x^*_{i}<x^*_1$. In particular, this condition implies that there is a unique coefficient $y^*_{i}\in(0,1)$ such that $x^*_{i}=y^*_{i} x^*_1$.

To close the section, we fix $N=5$ and compute numerically the optimal stopping thresholds $x^*_i$ for $i=1,\dots,5$. The parameter configuration reads as $r=0.05$, $\mu=0.008$, $\sigma=0.125$, $\lambda=0.1$ and $K=2$. In this case, the threshold $\hat{x}_\infty=\frac{\beta K}{\beta-1}\approx 2.593508$.

\begin{table}[h!]
\begin{center}
\begin{tabular}{|ccccc|}
\hline
$x^*_1$   &       $x^*_2$        &       $x^*_3$         &       $x^*_4$       &       $x^*_5$          \\
\hline
3.317653 & 3.079880 & 2.971528 & 2.738782 & 2.643230  \\
\hline
\end{tabular}
\vspace{0.1in}
\caption{The optimal exercise threshold $x^*_i$, $i=1,\dots,5$ under the parameter configuration $r=0.05$, $\mu=0.008$, $\sigma=0.125$, $\lambda=0.1$ and $K=2$.}
\end{center}
\end{table}

We observe from Table 1 that the thresholds $x^*_i$, $i=1,\dots,5$, form a decreasing sequence as a function of the number of stopping times left. This is in line with our general theory. Furthermore, Table 1 indicates that the thresholds $x^*_i$ converge to the threshold $\hat{x}_\infty$ as the number of stopping times left increases. This observation is also in line with our general theory.

\subsubsection*{Acknowledgements}
Jukka Lempa acknowledges financial support from the project "Energy markets: modelling, optimization and simulation (EMMOS)", funded by the Norwegian Research Council under grant 205328.

\end{document}